\documentclass[10pt]{amsart}
\usepackage{epsfig}
\usepackage{latexsym}
\usepackage{amsmath,amssymb,amsfonts,amsthm}
\usepackage{eucal}
\usepackage{eufrak}
\usepackage[all]{xypic}      
\usepackage{xspace}

\theoremstyle{plain}
\newtheorem{theorem}{Theorem}

\newtheorem{lemma}{Lemma}
\newtheorem{proposition}{Proposition}

\theoremstyle{definition}

\theoremstyle{remark}

\numberwithin{equation}{section}

\begin{document}

\title[Large components in random induced 
       subgraphs of n-cubes]
      {Large components in random induced 
       subgraphs of n-cubes}
\author{Christian M. Reidys}
\address{Center for Combinatorics, LPMC-TJKLC \\
         Nankai University  \\
         Tianjin 300071\\
         P.R.~China\\      
         Phone: *86-22-2350-5133-6800\\
         Fax:   *86-22-2350-9272}
\email{reidys@nankai.edu.cn}
\thanks{}
\date{January, 2008}
\begin{abstract}
In this paper we study random induced subgraphs of the binary
$n$-cube, $Q_2^n$. This random graph is obtained by selecting each 
$Q_2^n$-vertex with independent probability $\lambda_n$. Using a novel 
construction of subcomponents we study the largest component for 
$\lambda_n=\frac{1+\chi_n}{n}$, where $\epsilon\ge \chi_n\ge n^{-\frac{1}{3}+
\delta}$, $\delta>0$.
We prove that there exists a.s.~a unique largest component $C_n^{(1)}$.
We furthermore show that $\chi_n=\epsilon$,
$\vert C_n^{(1)}\vert\sim \alpha(\epsilon)\,\frac{1+\chi_n}{n}\, 2^n$ 
and for $o(1)=\chi_n\ge n^{-\frac{1}{3}+\delta}$, $\vert C_n^{(1)}\vert
\sim 2\,\chi_n\,\frac{1+\chi_n}{n}\, 2^n$ holds.
This improves the result of 
\cite{Bollobas:91} where constant $\chi_n=\chi$ is considered. 
In particular, in case of $\lambda_n=\frac{1+\epsilon}
{n}$, our analysis implies that a.s.~a unique giant 
component exists. 
\end{abstract} 
\maketitle
{{\small
}}

\section{Introduction}


\subsection{Background}
Burtin was the first \cite{Burtin:77} to study the connectedness of
random subgraphs of $n$-cubes, $Q_2^n$, obtained by selecting all 
$Q_2^n$-{\rm edges} independently (with probability $p_n$). 
He proved that a.s.~all such subgraphs are connected for $p>1/2$ and
are disconnected for $p<1/2$. Erd\H{o}s and Spencer \cite{Erdoes:79} 
refined Burtin's result and, more importantly in our context, they 
conjectured that there exists a.s.~a giant component for 
$p_n=\frac{1+\epsilon}{n}$ and $\epsilon>0$. Their conjecture 
was proved by Ajtai, Koml\'os and Szemer\'edi \cite{Ajtai:82}
who established the existence of a giant 
component for $p_n=\frac{1+\epsilon}{n}$. Key ingredients in their
proof are Harper's isoperimetric inequality \cite{Harper:66b} and a two 
round randomization, used for showing the non existence of certain splits. 
Several variations including the analysis of the giant component in random 
graphs with given average degree sequence have been studied 
\cite{Aizenman:87,Mesh:86,Molloy:98}.
Bollob\'as, Kohayakawa and Luczak \cite{Bollobas:92} analyzed the
behavior for $\epsilon$ tending to $0$ and showed in particular that the
constant for the giant component for fixed $\epsilon>0$ coincides with the
probability of infinite survival of the associated Poisson branching 
process. Spencer {\it et al.} \cite{Spencer:03} refined their results,
using specific properties of the $n$-cube as for instance the isoperimetric 
inequality \cite{Harper:66b} and Ajtai {\it et al.}'s two round 
randomization idea.
Considerably less is known for random induced subgraphs of the $n$-cube
obtained by independently selecting each $Q_2^n$-vertex with probability
$\lambda_n$. The main result here is the paper of Bollob\'as {\it et.al.} 
who have shown in \cite{Bollobas:91} for constant $\chi$ that 
$C_n^{(1)}=(1+o(1))\kappa \chi \frac{1+\chi}{n} 2^n$. In this paper we
improve this result. We show that for $\chi_n\ge n^{-\frac{1}{3}+\delta}$,
where $\delta>0$ a unique largest component exists and determine its size.
The key observation is a novel construction for small subcomponents
given in Lemma~\ref{L:poly}.

Random induced subgraphs arise in the context of molecular folding maps 
\cite{Schuster:94} where the neutral networks of molecular structures 
can be modeled as random induced subgraphs of $n$-cubes \cite{Reidys:97a}. 
They also occur in the context of neutral evolution of 
populations (i.e.~families of $Q_2^n$-vertices) consisting of erroneously 
replicating bit strings. Here, we work of course in $Q_4^n$, since we have
the alphabet $\{{\bf A},{\bf U},{\bf G},{\bf C}\}$.
Random induced subgraphs of $n$-cubes have had impact on conceptual level 
\cite{Schuster:02} and led to experimental work identifying sequences that 
realize two distinct ribozymes \cite{Schultes:00}.
A systematic computational analysis of neutral networks of molecular folding 
maps can be found in \cite{Gruener:95a}. The main result of this paper is 
the following

{\bf Theorem.}{\it 
$\,$
Let $Q_{2,\lambda_n}^n$ be the random graph consisting of $Q_2^n$-subgraphs,
$\Gamma_n$, induced by selecting each $Q_2^n$-vertex with independent 
probability $\lambda_n=\frac{1+\chi_n}{n}$, where $\epsilon\ge 
\chi_n\ge n^{-\frac{1}{3}+\delta}$, $\epsilon,\delta>0$. Then we have
\begin{eqnarray}
\lim_{n\to\infty}\mathbb{P}(\,  
\vert C_n^{(1)}\vert \sim \alpha(\epsilon)\,\frac{1+\epsilon}{n}\,2^n
\ \text{\it and $C_n^{(1)}$ is unique}\, )= 1 
\end{eqnarray}
and for $o(1)=\chi_n\ge n^{-\frac{1}{3}+\delta}$
\begin{eqnarray}
\lim_{n\to\infty}\mathbb{P}(\,  
\vert C_n^{(1)}\vert \sim 2\,\chi_n\,\frac{1+\chi_n}{n}\,2^n
\ \text{\it and $C_n^{(1)}$ is unique}\, )= 1. 
\end{eqnarray}
}
For $\chi_n=\epsilon$ the above theorem (combined with a straightforward 
argument for $\lambda_n\le\frac{1-\epsilon}{n}$) implies 
\begin{equation}
\lim_{n\to\infty}\mathbb{P}(\Gamma_n \,\text{\it has an unique giant 
component})=
\begin{cases}
1 & \text{\rm for $\lambda_n\ge\frac{1+\epsilon}{n}$ }\\
0 & \text{\rm for $\lambda_n\le\frac{1-\epsilon}{n}$ } \ .
\end{cases} 
\end{equation}
This is the random induced subgraph analogue of Ajtai {\it et al.}'s 
\cite{Ajtai:82} result. We present in Lemma~\ref{L:poly} a novel 
construction of subcomponents using branching processes inductively. 
We prove the main result using a generic vertex-boundary result due 
to Aldous \cite{Aldous:87,Babai:91b}. All results proved in this paper
remain valid for $n$-cubes over arbitrary, finite alphabets.

\subsection{Notation and terminology}
The binary $n$-cube, $Q_2^n$, is a combinatorial graph with vertex set
$\mathbb{F}_2^n$ in which two vertices are adjacent if they differ in 
exactly one coordinate. Let $d(v,v')$ be the number of 
coordinates by which $v$ and $v'$ differ. We set
\begin{eqnarray}\label{E:ball}
\forall \, A\subset \mathbb{F}_2^n,\, j\le n;\quad 
\text{\sf B}(A,j) & = & \{v\in\mathbb{F}_2^n\mid \exists \,\alpha\in A;\,
d(v,\alpha)\le j\} \\
\text{\sf S}(A,j) & = & \{v\in\mathbb{F}_2^n\mid \exists \,\alpha\in A;\,
d(v,\alpha)= j\} \\
\forall\, A\subset \mathbb{F}_2^n;\qquad 
\text{\sf d}(A) & = &\{v\in\mathbb{F}_2^n\mid \exists \, \alpha\in A;\, 
d(v,\alpha)=1\} 
\end{eqnarray}
and call $\text{\sf B}(A,j)$ and $\text{\sf d}(A)$ the ball of radius $j$
around $A$ and the vertex boundary of $A$ in $Q_2^n$, respectively.
If $A=\{\alpha\}$ we simply write $\text{\sf B}(\alpha,j)$.
Let $A,B\subset \mathbb{F}_2^n$, we call $A$ $\ell$-dense in $B$ if
$\text{\sf B}(v,\ell)\cap A\neq\varnothing$ for $v\in B$.
$Q_2^n$ can be viewed as the Cayley graph 
$\text{\sf Cay}(\mathbb{F}_2^n,\{e_i\mid i=1,\dots, n\})$ where $e_i$
is the canonical base vector. We will view $\mathbb{F}_2^n$ as a 
$\mathbb{F}_2$-vectorspace and denote the linear hull over 
$\{v_1,\dots,v_h\}$, $v_j\in\mathbb{F}_2^n$ by 
$\langle v_1,v_2,\dots,v_h\rangle$. 
There exists a natural linear order $\le$ over $Q_2^n$ given by
\begin{equation}\label{E:order}
v \le v'  \quad \Longleftrightarrow \quad 
(d(v,0) <d(v',0)) \ \vee \ 
(d(v,0)=d(v',0) \ \wedge \ v
<_{\text{\rm lex}} v' )  \ ,
\end{equation}
where $<_{\text{\rm lex}}$ denotes the lexicographical order. Any notion 
of minimal element or smallest element in $A\subset Q_2^n$ is considered
w.r.t.~the linear order $\le$ of eq.~(\ref{E:order}).

Each $A\subset \mathbb{F}_2^n$ induces a unique induced subgraph in $Q_2^n$, 
denoted by $Q_2^n[A]$, in which $a_1,a_2\in A$ are adjacent iff 
$a_1,a_2$ are adjacent in $Q_2^n$. 
Let $Q_{2,\lambda_n}^n$ be the random graph consisting of 
$Q_2^n$-subgraphs, $\Gamma_n$, induced by selecting each $Q_2^n$-vertex 
with independent probability $\lambda_n$. 
That is, $Q_{2,\lambda_n}^n$ is the finite probability space 
$(\{Q_2^n[A]\mid A\subset \mathbb{F}_2^n\},\mathbb{P})$, 
with the probability measure 
$\mathbb{P}(A)=\lambda_n^{\vert A\vert}\,
(1-\lambda_n)^{2^n-\vert A\vert}$. A property $\text{\sf M}$ is a 
subset of induced subgraphs of $Q_2^n$ closed 
under graph isomorphisms. The terminology ``$\text{\sf M}$ holds 
a.s.'' is equivalent to $\lim_{n\to\infty}\mathbb{P}(\text{\sf M})=1$.
A component of $\Gamma_n$ is a maximal connected induced 
$\Gamma_n$-subgraph, $C_n$. 
The largest $\Gamma_n$-component is denoted by $C_n^{(1)}$. 
It is called a giant component if and only if
\begin{equation}
\exists \, \kappa>0,\quad \vert C_n^{(1)}
\vert \ge \kappa\, \vert \Gamma_n\vert \ ,
\end{equation}
and $x_n\sim y_n$ is equivalent to (a) $\lim_{n\to\infty}x_n/y_n$ exists
and (b) $\lim_{n\to\infty}x_n/y_n=1$.
Let $Z_n=\sum_{i=1}^n \xi_i$ be a sum of mutually independent indicator 
random variables (r.v.), $\xi_i$ having values in $\{0,1\}$. Then we have, 
\cite{Chernoff:52}, for $\eta>0$ and 
$c_\eta=\min\{-\ln(e^{\eta}[1+\eta]^{-[1+\eta]}), \frac{\eta^2}{2}\}$ 
\begin{equation}\label{E:cher}
\text{\sf Prob}(\,\vert\,Z_n-\mathbb{E}[Z_n]\,\vert\,>
\eta\,\mathbb{E}[Z_n]\,) \le 
        2 e^{-c_\eta \mathbb{E}[Z_n]}\, .
\end{equation}
$n$ is always assumed to be sufficiently large and $\epsilon$ is a
positive constant satisfying $0<\epsilon < \frac{1}{3}$. We use the 
notation $B_m(\ell,\lambda_n)=\binom{m}{\ell}\lambda_n^\ell\,
(1-\lambda_n)^{m-\ell}$ and write 
$g(n)=O(f(n))$ and $g(n)=o(f(n))$ for $g(n)/f(n)\to \kappa$ as
$n\to \infty$ and $g(n)/f(n)\to 0$ as $n\to \infty$, respectively.

\section{Preliminaries}
Let us briefly recall some basic facts about branching processes 
\cite{Harris:63,Kolchin:86}. Suppose $\xi$ is a random variable and 
$(\xi_i^{(t)})$, $i,t\in\mathbb{N}$ counts the number of offspring 
of the $i$th-individual at generation $t-1$. 
We consider the family of r.v.~$(Z_i)_{i\in \mathbb{N}_0}$: $Z_0=1$ and 
$Z_{t} = \sum_{i=1}^{Z_{t-1}}\xi_i^{(t)}$ for $t\ge 1$ and
interpret $Z_t$ as the number of individuals ``alive'' in generation 
$t$. We will be interested in the limit probability
$\lim_{t\to\infty}\text{\sf Prob}(Z_t>0)$, i.e.~the probability of 
infinite survival. We have
\begin{theorem}\label{T:galton}
Let $u_n=n^{-\frac{1}{3}}$, $\lambda_n=\frac{1+\chi_n}{n}$,
$m=n-\lfloor \frac{3}{4}u_n n\rfloor$ and $\text{\sf Prob}(\xi=\ell)=
B_m(\ell,\lambda_n)$. Then for $\chi_n=\epsilon$ the r.v.~$\xi$ becomes 
asymptotically Poisson, i.e.~ $\mathbb{P}(\xi=\ell) \sim 
\frac{(1+\epsilon)^\ell}{\ell!}\, 
e^{-(1+\epsilon)}$ and 
\begin{equation}
0<\lim_{t\to\infty}\text{\sf Prob}(Z_t>0)=\alpha(\epsilon)<1 \ .
\end{equation}
For $o(1)=\chi_n\ge n^{-\frac{1}{3}+\delta}$, $\delta>0$ we have
\begin{equation}
\lim_{t\to\infty}\text{\sf Prob}(Z_t>0)=(2+o(1))\,\chi_n \ .
\end{equation}
\end{theorem}
In view of Theorem~\ref{T:galton} we introduce the notation
\begin{eqnarray}
\pi(\chi_n) & = & 
\begin{cases}
\alpha(\epsilon)            & \quad \text{\rm for}\quad \chi_n=\epsilon \\
2(1+o(1))\chi_n  & \quad \text{\rm for}\quad   \ o(1)=\chi_n\ge 
                      n^{-\frac{1}{3}+\delta} \ .
\end{cases}
\end{eqnarray}
We procceed by labeling the indices of a $Q_2^n$-vertex 
$v=(x_1,\dots,x_n)$. For this purpose set
\begin{equation}\label{E:nu}
\nu_n  =  \lfloor \frac{1}{2k(k+1)}u_nn\rfloor,\quad
\iota_n  =  \lfloor\frac{k}{2k+1}u_nn\rfloor,\quad\text{\rm and}\quad
z_n=k\nu_n+\iota_n \ .
\end{equation} 
We write an $Q_2^n$-vertex $v=(x_1,\dots,x_n)$ as
\begin{equation}\label{E:not-seq}
   (\underbrace{x_1^{(1)},\dots,x^{(1)}_{\nu_n}}_{\nu_n\, {\rm coordinates}},
   \underbrace{x_1^{(2)},\dots,x^{(2)}_{\nu_n}}_{\nu_n\, {\rm coordinates}},
    \dots,
 \underbrace{x_1^{(k+1)},\dots,x^{(k+1)}_{\iota_n}}_{
     \iota_n\, {\rm coordinates}},
   \underbrace{x_{u_n +1},\dots,x_{n}}_{n-z_n
        \ge \atop n-\lfloor \frac{1}{2}u_n n\rfloor \, {\rm coordinates}}) \  .
\end{equation}
For any $1\le s\le \nu_n$, $r=1,\dots,k$ we set $e_s^{(r)}$ to be
the $s + (r-1)\nu_n$th-unit vector, i.e. $e_s^{(r)}$ has exactly one $1$ at 
its $(s+ (r-1)\nu_n)$th coordinate. Similarly let $e_s^{(k+1)}$, 
$1\le s\le \iota_n$ denote the $(s + k\nu_n)$th-unit vector.
We use the standard notation for the $z_n +1\le t\le n$ 
unit vectors, i.e. $e_t$ is the vector where $x_t=1$ and $x_j=0$, otherwise.

In our first lemma we use Theorem~\ref{T:galton} in order to obtain 
information about small components in $\Gamma_n$.
\begin{lemma}\label{L:G}
Suppose $\lambda_n=\frac{1+\chi_n}{n}$ and $\epsilon\ge \chi_n\ge 
n^{-\frac{1}{3}+\delta}$, where $\delta>0$. 
Then each $\Gamma_n$-vertex is contained in a $\Gamma_n$-subcomponent 
of size $\lfloor \frac{1}{4}u_n n \rfloor$ with probability at least 
$\pi(\chi_n)$.
\end{lemma}
\begin{proof}
We consider a branching-process in the sub-cube $Q_2^{n-z_n}$ 
(eq.~(\ref{E:not-seq})). 
W.l.o.g.~we initialize the process at $v=(0,\dots,0)$ and set 
$E_0=\{e_{n-z_n+1},\dots,e_n\}$ and $L_0^{(0)}=\{(0,\dots,0)\}$.
We consider the $n-\lfloor  \frac{3}{4}u_nn\rfloor$ smallest neighbors of $v$. 
Starting with the smallest we select each of them with 
independent probability $\lambda_n=\frac{1+\chi_n}{n}$. 
Suppose $v+e_j$ is the first
being selected. Then we set $E_1=E_0\setminus \{e_j\}$ and 
$L_1^{(0)}=L_0^{(0)}\cup \{e_j\}$ and proceed inductively setting 
$E_s=E_{s-1}\setminus\{e_{w}\}$ and $L_t^{(0)}=L_{t-1}^{(0)}\cup 
\{e_w\}$ for each neighbor $v+e_{w}$ being selected. We are given 
the following alternative, either we have (a)
$\vert E_s\vert=n-(\lfloor\frac{3}{4}\,u_nn\rfloor-1)$ or
(b) $\vert E_s\vert > n-(\lfloor\frac{3}{4}\,u_nn\rfloor-1)$ and 
all smallest $n-\lfloor\frac{3}{4}\,u_nn\rfloor$ neighbors of 
    $(0,\dots,0)$ are checked.
In case of (a) $\lfloor\frac{1}{4}\,u_nn\rfloor-1$ vertices have been 
connected and since
$$ 
n-z_n- (\lfloor\frac{1}{4}\,u_nn\rfloor -1)\ge n-\frac{1}{2}u_n n
 -\frac{1}{4}u_nn +1 \ge n-\lfloor\frac{3}{4}u_nn\rfloor 
$$
there are still $\ge n-\lfloor\frac{3}{4}u_nn\rfloor$ neighbors available. 
Suppose $\vert E_s\vert>n-(\lfloor\frac{3}{4}u_nn\rfloor-1)$ and all smallest 
$n-\lfloor\frac{3}{4}\,u_nn\rfloor$ neighbors of $(0,\dots,0)$ were 
examined. Then we proceed by choosing the smallest element of 
$L_{t_0}^{(0)}\setminus \{0\}$, $v_1^*$ and set 
$L_0^{(1)}=L_{t_0}^{(0)}\setminus \{v_1^*\}$. By construction, $v_1^*$ 
has at least $n-\lfloor\frac{3}{4}u_nn\rfloor$ neighbors of the form 
$v_1^*+e_r$ $e_r\in E_s$. 
We begin with the smallest of these and continue selecting with 
probability $\frac{1+\chi_n}{n}$ setting $E_s=E_{s-1}\setminus\{e_{j}\}$ 
and $L_t^{(1)}=L_{t-1}^{(1)}\cup\{ v_1^*+e_j\}$ for each neighbor 
$v_1^*+e_{j}$ being selected. 
We continue inductively setting $L_0^{(r)}=L_{t_{r-1}}^{(r-1)}\setminus 
\{v_r^*\}$ and stop in case of (a).
By construction, this process yields an induced sub-tree of $Q_2^{n-z_n}$. 
Theorem~\ref{T:galton} guarantees that we have a $\Gamma_n$-subcomponent 
of size $\lfloor \frac{1}{4}u_n n\rfloor$ with probability at least
$\pi(\chi_n)$. 
\end{proof}
We refer to the particular branching process used in Lemma~\ref{L:G} as 
$\gamma$-process. The $\gamma$-process produces a subcomponent
of size $\lfloor \frac{1}{4}u_n n\rfloor$ , which we refer to as
$\gamma$-(sc).


\section{Small subcomponents}

The $\gamma$-process employed in Lemma~\ref{L:G} did by construction not 
involve the first $z_n$ coordinates. In the following lemma we will use 
the first $k\,\nu_n$ of them in order to build inductively larger 
subcomponents (sc). 
\begin{lemma}\label{L:poly}
Let $k\in\mathbb{N}$ be arbitrary but fixed, $\lambda_n=
\frac{1+\chi_n}{n}$, $\nu_n=\lfloor\frac{u_nn}{2k(k+1)}\rfloor$
and $\varphi_n=\pi(\chi_n)\nu_n (1-e^{-(1+\chi_n)u_n/4})$.
Then there exists $\rho_k>0$ such that each $\Gamma_n$-vertex is 
with probability at least 
\begin{equation}\label{E:ist}
\pi_k(\chi_n)=
\pi(\chi_n)\,\left(1-e^{-\rho_k \varphi_n}
\right)
\end{equation}
contained in a $\Gamma_n$-subcomponent of size at least
$ c_k \,(u_nn)\varphi_n^k$, where $c_k>0$.
\end{lemma}
Lemma~\ref{L:poly} gives rise to introduce the induced subgraph 
$\Gamma_{n,k}=Q_2^n[A]$ where
\begin{equation}\label{E:'}
A=\{v\mid \text{\rm $v$ is contained in a 
$\Gamma_n$-(sc) of size $\ge c_k\, (u_nn)\varphi_n^k$, $c_k>0$}\} \ .
\end{equation}
In case of 
$\epsilon \ge \chi_n\ge n^{-\frac{1}{3}+\delta}$ 
we have $1-e^{-\frac{1}{4}(1+\chi_n)u_n}\ge \,u_n/4$ 
and consequently 
$\varphi_n \ge c'\,(1+o(1))\chi_n u_n^2\,n\ge c_0\,n^{\delta}$ for 
some $c',c_0>0$. Furthermore
\begin{equation}\label{E:cool2}
\lfloor \frac{1}{4}u_nn\rfloor \, \varphi_n^k\ge  
c_k \, n^{\frac{2}{3}} n^{k\delta}, 
\quad c_k>0 \ .
\end{equation}
Accordingly, choosing $k$ sufficiently large, each $\Gamma_n$-vertex is 
contained in a (sc) of arbitrary polynomial size with probability at least
\begin{equation}
\pi(\chi_n)\,\left(1-e^{-\rho_k n^{\delta}}\right), 
\quad 0<\delta  , \ 0<\rho_k\ .
\end{equation}
\begin{proof}
Since all translations are $Q_2^n$-automorphisms we can w.l.o.g.~assume that 
$v=(0,\dots,0)$. We use the notation of eq.~(\ref{E:not-seq}) and recruit 
the $n-z_n$-unit vectors $e_t$ for a $\gamma$-process. 
The $\gamma$-process of Lemma~\ref{L:G} yields a $\gamma$-(sc), $C(0)$, of 
size $\lfloor \frac{1}{4} u_n n\rfloor$ with probability $\ge\pi(\chi_n)$.  
We consider for $1\le i\le k$ the sets of $\nu_n$ elements 
$B_i=\{e_1^{(i)},\dots, e_{\nu_n}^{(i)}\}$ and set 
$H=\langle e_{u_n+1},\dots,e_n\rangle$. By construction we have
\begin{equation}\label{E:linear}
\langle B_i\cup \langle \bigcup_{1\le j\le i-1}B_j\rangle \oplus H\rangle=
\langle B_i\rangle \oplus\langle \bigcup_{1\le j\le i-1}B_j\rangle \oplus H \ .
\end{equation}
In particular, for any $1\le s<j\le \nu_n$: $e_s^{(1)}-e_j^{(1)}\in H$ 
is equivalent to $e_s^{(1)}=e_j^{(1)}$. Since all vertices are selected 
independently and $\vert C(0)\vert=\lfloor\frac{1}{4}u_n n
\rfloor$, for fixed $e_s^{(1)}\in B_1$ the 
probability of not selecting a vertex $v'\in e_s^{(1)}+C(0)$ is given by
\begin{equation}\label{E:failure}
\mathbb{P}\left(\left\{e_s^{(1)}+\xi\mid \xi\in C(0)\right\}\cap\Gamma_n=
\varnothing\right)=
\left(1-\frac{1+\chi_n}{n}\right)^{\lfloor\frac{1}{4}u_nn
\rfloor}\sim 
e^{-(1+\chi_n)\frac{1}{4}u_n} \ .
\end{equation}
We set $\mu_n =(1-e^{-(1+\chi_n)\frac{1}{4}u_n})$, i.e. 
$\mu_n=
\mathbb{P}\left((e_s^{(1)}+C(0))\cap \Gamma_n \neq \varnothing\right)$ and 
introduce the r.v.
\begin{equation}
X_1=
\left| \left\{ e_s^{(1)}\in B_1 \mid \exists\, \xi\in C(0);\, 
         e_s^{(1)}+\xi\in \Gamma_n \right\} \right| \ .
\end{equation}
Obviously, $\mathbb{E}(X_1)=\mu_n \nu_n$ and 
using the large deviation result of eq.~(\ref{E:cher}) we can conclude that 
\begin{equation}
\exists \,\rho>0;\quad 
\mathbb{P}\left(X_1 < \frac{1}{2} \mu_n\nu_n \right)\le
e^{-\rho\,\mu_n \nu_n} \ .
\end{equation}
Suppose for $e_s^{(1)}$ there exists some $\xi\in C(0)$ such that
$e_s^{(1)}+\xi\in \Gamma_n$ (that is $e_s^{(1)}$ is counted by $X_1$). 
We then select the smallest
element of the set $\{e_s^{(1)}+\xi\mid 
\xi\in C(0)\}$, say $e_s^{(1)}+\xi_0$ and initiate 
a $\gamma$-process using the $n-z_n$ elements 
$\{e_{z_n+1},\dots,e_n\}$ at $e_s^{(1)}+\xi_0$.
The process yields a $\gamma$-(sc) of size 
$\lfloor\frac{1}{4}u_nn\rfloor$ with probability at least $\pi(\chi_n)$. 
For any two elements $e_s^{(1)},e_j^{(1)}$ with $e_s^{(1)}+\xi(e_s^{(1)}),
e_j^{(1)}+\xi(e_j^{(1)})\in\Gamma_n$ the respective sets are vertex 
disjoint since $\langle B_1 \cup H\rangle=\langle B_1\rangle\oplus
H$.
Let $\tilde{X}_1$ be the random variable counting the number of these new, 
pairwise vertex disjoint sets of $\gamma$-(sc) of size 
$\lfloor \frac{1}{4}u_n n\rfloor$.
By construction each of them is connected to $C(0)$.
We immediately observe $\mathbb{E}(\tilde{X}_1)\ge \pi(\chi_n)\mu_n\nu_n$ 
and set $\varphi_n=\pi(\chi_n)\mu_n\nu_n$. Using the large deviation 
result in eq.~(\ref{E:cher}) we derive 
\begin{equation}
\exists \,\rho_1>0;\quad 
\mathbb{P}\left(\tilde{X}_1 < \frac{1}{2} \varphi_n\right)\le 
e^{-\rho_1\varphi_n} \ .
\end{equation}
We proceed by proving that for each $1\le i\le k$ there exists a sequence
of r.v.s $(\tilde{X}_1,\tilde{X}_2,\dots,\tilde{X}_i)$ where $\tilde{X}_i$ 
counts the number of pairwise disjoint sets of $\gamma$-(sc) added at 
step $1\le j\le i$ such that:\\
(a) all sets, $C_\alpha^{(j)}$, $1\le j\le i$, added until step $i$ are 
    pairwise vertex disjoint and are of size $\lfloor \frac{1}{4}u_nn\rfloor$\\
(b) all sets added until step $i$ are connected to $C(0)$ and
\begin{equation}
\exists \,\rho_i>0;\quad 
\mathbb{P}\left( \tilde{X}_i< \frac{1}{2^i}(\varphi_n)^i\right) \le
e^{-\rho_i\varphi_n} \ , \ \text{\rm where}\ \varphi_n=
\pi(\chi_n)\mu_n\nu_n \ .
\end{equation}
We prove the assertion by induction on $i$. Indeed in our construction of 
$\tilde{X}_1$ have already established the induction basis. 
In order to define $\tilde{X}_{i+1}$ we use the set $B_{i+1}=
\{e_{1}^{(i+1)},\dots,e_{\nu_n}^{(i+1)}\}$. For each
$C_\alpha^{(i)}$ counted by $\tilde{X}_i$ (i.e.~the vertices that were 
connected in step $i$) we form the set $e_s^{(i+1)}+C_\alpha^{(i)}$. 
By induction hypothesis two different $C_\alpha^{(i)},C_{\alpha'}^{(i)}$, 
counted by $\tilde{X}_i$ are vertex disjoint and connected to $C(0)$. 
Since $\langle B_{i+1}\rangle\bigoplus \langle \bigcup_{1\le j\le 
i}B_j\rangle \bigoplus H$ are disjoint we can conclude 
$$
 (s\neq s'\; \vee\; \alpha\neq \alpha)\quad \Longrightarrow\quad 
(e_s^{(i+1)}+C_\alpha^{(i)}) \cap 
(e_{s'}^{(i+1)}+C_{\alpha'}^{(i)})=\varnothing
$$ 
and the probability that we have 
for fixed $C_\alpha^{(i)}$: $(e_s^{(i+1)}+C_\alpha^{(i)})\cap\Gamma_n=
\varnothing$ for some $e_s^{(i+1)}\in B_{i+1}$ is 
exactly as in eq.~(\ref{E:failure})
$$
\mathbb{P}\left((e_s^{(i+1)}+C_\alpha^{(i)})\cap\Gamma_n=\varnothing\right)=
\left(1-\frac{1+\chi_n}{n}\right)^{\lfloor\frac{1}{4}u_nn\rfloor}
\sim e^{-(1+\chi_n)\frac{1}{4}u_n}\ .
$$
As for the induction basis, $\mu_n =(1-e^{-(1+\chi_n)\frac{1}{4}u_n})$ is the 
probability that $(e_s^{(i+1)}+C_\alpha^{(i)})\cap\Gamma_n\neq \varnothing$. 
We proceed by defining the r.v.
\begin{equation}
X_{i+1}=\sum_{C_\alpha^{(i)}}
\left| \left\{ e_s^{(i+1)}\in B_{i+1} \mid \exists\, \xi\in C_\alpha^{(i)};\, 
         e_s^{(i+1)}+\xi\in \Gamma_n \right\} \right| \ .
\end{equation}
$X_{i+1}$ counts the number of events where 
$(e_s^{(i+1)}+C_\alpha^{(i)})\cap\Gamma_n\neq \varnothing$ for each 
$C_\alpha^{(i)}$, respectively. For fixed $C_\alpha^{(i)}$ and fixed 
$e_s^{(i+1)}\in B_{i+1}$ we choose the minimal element 
$$
e_s^{(i+1)}+\xi_{0,\alpha}\in \left\{ e_s^{(i+1)}+\xi_{\alpha}\mid  
\xi_\alpha\in C_\alpha^{(i)},\, e_s^{(i+1)}+\xi_\alpha\in \Gamma_n \right\}\ .
$$
Then $X_{i+1}$ counts exactly the minimal elements 
$e_s^{(i+1)}+\xi_{0,\alpha},e_{s'}^{(i+1)}+\xi_{0,\alpha'},\dots$ 
for all $C_\alpha^{(i)},{C}_{\alpha'}^{(i)},\dots$
and any two can be used to construct pairwise vertex disjoint $\gamma$-(sc)
of size $ \lfloor\frac{1}{4}u_nn\rfloor$.
We next define $\tilde{X}_{i+1}$ to be the r.v.~counting the number
of events that the $\gamma$-process in $H$ initiated at the
$e_s^{(i+1)}+\xi_{0,\alpha}\in \Gamma_n$ yields a $\gamma$-(sc) of size 
$\lfloor \frac{1}{4}u_nn\rfloor$. By construction each of these 
is connected to a unique $C_\alpha^{(i)}$.
Since $\langle B_{i+1}\rangle \bigoplus \langle \bigcup_{1\le j\le i}
B_j\rangle \bigoplus H$ all newly added sets are pairwise vertex disjoint 
to all previously added vertices. We derive
\begin{eqnarray*} 
\mathbb{P}\left(\tilde{X}_{i+1} < \frac{1}{2^{i+1}} 
\varphi_n^{i+1}\right)  
& \le & 
\underbrace{\mathbb{P}\left( \tilde{X}_i <
 \frac{1}{2^i}\varphi_n^i\right)}_{\text{\rm failure at step $i$}}
\ +\underbrace{
\mathbb{P}\left(\tilde{X}_{i+1} <\frac{1}{2^{i+1}} 
\varphi_n^{i+1} \, \wedge \, 
\tilde{X}_i \ge \frac{1}{2^i} \varphi_n^i\right)
}_{\text{\rm failure at step $i+1$ conditional to $\tilde{X}_i\ge 
\frac{1}{2^i}\varphi_n^i$}} \\
 & \le & e^{-\rho_i\,\varphi_n} + 
e^{-\rho \,\varphi_n^{i+1}} 
(1-e^{-\rho_i\, \varphi_n})\, , \quad \rho>0 \\
& \le & e^{-\rho_{i+1}\,\varphi_n} \ .
\end{eqnarray*}
Therefore each $\Gamma_n$-vertex is with probability at least 
$\pi(\chi_n) \,(1-e^{-\rho_k\varphi_n})$ contained 
in a $\Gamma_n$-(sc) of size at least $c_k\, (\chi_nn)\varphi_n^k$, 
for $c_k>0$ and the lemma is proved.
\end{proof}
We next prove a technical lemma which will be instrumental for the proof
of Lemma~\ref{L:size}.
We show that the number of vertices not contained in $\Gamma_{n,k}$ 
is sharply concentrated, using a strategy
similar to that in Bollob\'as~{\it et.al.} \cite{Bollobas:92}. 
Let $U_n$ denote the complement of $\Gamma_{n,k}$ in $\Gamma_n$.

\begin{lemma}\label{L:mini}
Let $k\in\mathbb{N}$ and $\lambda_n=\frac{1+\chi_n}{n}$, where 
$\epsilon\ge \chi_n\ge n^{-\frac{1}{3}+\delta}$. 
Then we have
\begin{equation}
\mathbb{P}\left(\vert\,\vert U_n \vert -\mathbb{E}[\vert U_n\vert]\,\vert
\ge \frac{1}{n} \mathbb{E}[\vert U_n\vert]\right)=o(1) \ .
\end{equation}
\end{lemma}
\begin{proof}
Let $X_v$ be the indicator variable for the event $v\in U_n$, i.e.~$v$ is
contained in a $\Gamma_n$-(sc) of size $<c_k \,(u_nn)\varphi_n^k$.
Clearly $\vert U_n\vert =\sum_{v\in \Gamma_n}X_v$. In order to prove our 
concentration result we proceed by estimating $\mathbb{V}[\vert U_n\vert]$.
Suppose $v\neq v'$. There are two ways by which 
$X_v,X_{v'}$ viewed as r.v.~over $Q_{2,\lambda_n}^n$ can be correlated.
First $v,v'$ can belong to the same component in which case we write 
$v\sim_1 v'$. Clearly,
\begin{equation}
\sum_{v\sim_1 v'}\mathbb{E}[X_v\,X_{v'}] \le c_k \,(u_nn)\varphi_n^k\; \,
\mathbb{E}[\vert U_n\vert]\ .
\end{equation}
Secondly, correlation arises when $v,v'$ belong to two different components 
$C_v$, $C_{v'}$ having minimal distance $2$ in $Q_{2}^n$. In this case we 
write $v\sim_2 v'$. Then there
exists some $Q_2^n$-vertex, $w$, such that
$\{w\}\subset {\sf d}(C_v)\cap {\sf d}(C_{v'})$ and we derive
\begin{eqnarray*}
\mathbb{P}(\text{\rm $C_v$ and $C_{v'}$ are $\Gamma_n$-(sc),
$d(C_v,C_{v'})=2$}) & = & \frac{1-\lambda_n}{\lambda_n}\,
\mathbb{P}(C_v\cup C_{v'}\cup \{w\} \ \text{\rm is a $\Gamma_n$-(sc)}) \\
& \le & n \, \mathbb{P}(C_v\cup C_{v'}\cup
\{w\} \ \text{\rm is a $\Gamma_n$-(sc)}) \ .
\end{eqnarray*}
Since for $v\sim_2v'$  we have 
$\mathbb{P}(\text{\rm $C_v$ and $C_{v'}$ are $\Gamma_n$-(sc),
$d(C_v,C_{v'})=2$})=\mathbb{E}[X_v\,X_{v'}]$ we can 
immediately give the upper bound
\begin{equation}
\sum_{v\sim_2v'}\mathbb{E}[X_v\,X_{v'}]\le n\, (2c_k \,(u_nn)\varphi_n^k+1)^3
\, \vert\Gamma_n\vert \ .
\end{equation}
The uncorrelated pairs $(X_v,X_{v'})$, writing $v\not\sim v'$, can easily
be estimated by
\begin{equation}
\sum_{v\not\sim v'}\mathbb{E}[X_v\,X_{v'}]=
\sum_{v\not\sim v'}\mathbb{E}[X_v]\mathbb{E}[X_{v'}]\le \mathbb{E}[\vert U_n\vert]^2 \ .
\end{equation}
Since $\mathbb{V}[\vert U_n\vert]=\mathbb{E}[\vert U_n\vert (\vert U_n\vert
-1)]+\mathbb{E}[\vert U_n\vert]-\mathbb{E}[\vert U_n\vert]^2$ we have
\begin{eqnarray*}
\mathbb{E}[\vert U_n\vert(\vert U_n\vert -1)] & = & 
\sum_{v\sim_1 v'}\mathbb{E}[X_v\,X_{v'}]+
\sum_{v\sim_2 v'}\mathbb{E}[X_v\,X_{v'}] +
\sum_{v\not\sim v'}\mathbb{E}[X_v\,X_{v'}] \\
& \le & c_k \,(u_nn)\varphi_n^k\; \mathbb{E}[\vert U_n\vert] +
n\, (2c_k \,(u_nn)\varphi_n^k+1)^3 \vert\Gamma_n\vert + \mathbb{E}[\vert U_n\vert]^2 \ .
\end{eqnarray*}
Considering isolated vertices in $\Gamma_n$ immediately implies
$\mathbb{E}[\vert U_n\vert]\ge c \vert\Gamma_n\vert$ for some $c>0$, whence
$$
\mathbb{V}[\vert U_n\vert]/\mathbb{E}[\vert U_n\vert]^2 \le \frac{
(1+c_k \,(u_nn)\varphi_n^k)+ 
n\, (2c_k \,(u_nn)\varphi_n^k+1)^3}{c^2\,\vert\Gamma_n\vert}=o(\frac{1}{n^2})
\ .
$$
Via Chebyshev's inequality 
$\mathbb{P}(\vert \vert U_n\vert -\mathbb{E}[\vert U_n\vert]\vert\ge
\frac{1}{n}\, \mathbb{E}[\vert U_n\vert] )\le n^2\,
\frac{\mathbb{V}[\vert U_n\vert]}{\mathbb{E}[\vert U_n\vert]^2}$ holds, whence
the lemma. 
\end{proof}

\begin{lemma}\label{L:size}
Let $\lambda_n=\frac{1+\chi_n}{n}$ where $\epsilon\ge \chi_n\ge 
n^{-\frac{1}{3}+\delta}$. Then we have for sufficiently large 
$k\in\mathbb{N}$ 
\begin{equation}
(1-o(1))\,\pi(\chi_n)\,\vert \Gamma_n\vert\le 
\vert \Gamma_{n,k} \vert \le  (1+o(1))\, \pi(\chi_n) \, 
\vert\Gamma_n\vert \ \qquad \text{\it a.s.}
\end{equation}
\end{lemma}
In view of Lemma~\ref{L:poly} the crucial part is to show that there are 
sufficiently many $\Gamma_n$-vertices contained in $\Gamma_n$-(sc) 
of size $< c_k\, u_nn\varphi_n^k$. 
For this purpose we use a strategy introduced by Bollob\'as {\it et.al.} 
\cite{Bollobas:92} and consider the $n$-regular rooted tree $T_n$.
Let $v^*$ denote the root of $T_n$. Then $v^*$ has $n$ descendents and
all other $T_n$-vertices have $n-1$
Selecting the $T_n$-vertices with independent probability $\lambda_n$ 
we obtain the probability space $T_{n,\lambda_n}$ whose elements, $A_n$, 
are random induced subtrees. We will be interested in the $A_n$-component 
which contains the root, denoted by $C_{v^*}$.
Let $\xi_{v^*}$ and $\xi_v$, for $v\neq v^*$ be two r.v.~such that 
$\text{\sf Prob}(\xi_{v^*}=\ell)=B_n(\ell,\lambda_n)$ and $\text{\sf Prob}
(\xi_{v}=\ell)=B_{n-1}(\ell,\lambda_n)$, respectively. We assume that
$\xi_{v^*}$ and $\xi_v$ count the offspring produced at $v^*$ and 
$v\neq v^*$. Then the induced branching process initialized at 
$v^*$, $(Z_i)_{i\in \mathbb{N}_0}$ constructs $C_{v^*}$. Let $\pi_0(\chi)$ 
denote its survival probability, then we have in view of Theorem~\ref{T:galton}
and \cite{Bollobas:92}, Corollary $6$:
\begin{equation}\label{E:ll}
\pi_0(\chi_n) = (1+o(1))\,\pi(\chi_n)\ .
\end{equation}
\begin{proof}
{\it Claim $1$.} $\vert \Gamma_{n,k}\vert \ge 
\left((1-o(1))\,\pi(\chi_n)\right)\,\vert \Gamma_n\vert$ a.s.\\
According to Lemma~\ref{L:poly} we have 
$
\mathbb{E}[\vert U_n\vert] < (1-\pi_k(\chi_n))\,\vert \Gamma_n\vert
$
and we can conclude using Lemma~\ref{L:mini} and 
$\mathbb{E}[\vert U_n\vert]= O(\vert\Gamma_n\vert)$
\begin{equation}
\vert U_n\vert < \left(1+O(\frac{1}{n})\right) \,\mathbb{E}[\vert U_n\vert]
<\left(1-(\pi_k(\chi_n)-O(\frac{1}{n}))\right)\vert
\Gamma_n\vert \quad \text{\rm a.s.}
\end{equation}
In view of eq.~(\ref{E:ist}) and $\chi_n\ge n^{-\frac{1}{3}+\delta}$ we have 
for arbitrary but fixed $k$, 
$$
\pi_k(\chi_n)-O(\frac{1}{n})=(1-o(1))\,\pi(\chi_n)\ .
$$
Therefore we derive
\begin{equation}\label{E:jk0}
\vert\Gamma_{n,k}\vert \ge (1-o(1))\,\pi(\chi_n)\,\vert
\Gamma_n\vert \quad \text{\rm a.s.,}
\end{equation}
and Claim $1$ follows.\\
{\it Claim $2$.} For sufficiently large $k$, $\vert \Gamma_{n,k}
\vert \le \left((1+o(1))\,\pi(\chi_n)\right)\,\vert \Gamma_n\vert$ 
a.s.~holds.\\
For any fixed $Q_2^n$-vertex, $v$, we have the inequality
\begin{equation}
\mathbb{P}\left(\vert C_{v^*}\vert \le \ell\right)\le
\mathbb{P}\left(\vert C_{v}\vert \le \ell\right) \ .
\end{equation}
Indeed we can obtain $C_v$ by inductively constructing a spanning tree as 
follows: suppose the set of all $C_v$-vertices at distance $h$ is $M_h^{C_v}$. 
Starting with the smallest $w\in M_h^{C_v}$ ($h\ge 1$) there are at most 
$n-1$ $w$-neighbors contained in $M_{h+1}^{C_v}$ that are not neighbors for
some smaller $w'\in M_h^{C_v}$. Hence for any $w\in M_h^{C_v}$ at most 
$n-1$ vertices have to be examined. 
The $A_n$-component $C_{v^*}$ is generated by the same procedure. Then 
for each $w\in M_h^{C_{v^*}}$ there are exactly $n-1$ neighbors in 
$M_{h+1}^{C_{v^*}}$. 
Since the process adds at each stage less or equally many vertices for
$C_v$ we have by construction $\vert C_v\vert\le \vert C_{v^*}\vert$.
Standard estimates for Binomial coefficients allow to estimate the numbers 
of $T_n$-subtrees containg the root \cite{Bollobas:92}, Corollary $3$. 
Since vertex boundaries in $T_n$ are easily 
obtained we can accordingly compute 
$\mathbb{P}(\vert C_{v^*}\vert =\ell)$. Choosing $k$ sufficiently large 
the estimates in \cite{Bollobas:92}, Lemma $22$, guarantee
\begin{equation}\label{E:bas}
\mathbb{P}\left(\vert C_{v^*}\vert < c_k\, u_nn\, \varphi_n^k\right)=
(1-\pi_0(\chi_n))+o(e^{-n}) \ .
\end{equation}
In view of $\mathbb{P}\left(\vert C_{v^*}\vert \le \ell\right)\le
\mathbb{P}\left(\vert C_{v}\vert \le \ell\right)$ and eq.~(\ref{E:ll}) 
we can conclude from eq.~(\ref{E:bas}) 
\begin{equation}
 (1-(1+o(1))\pi(\chi_n))\,\vert \Gamma_n\vert+o(1)  \le 
\mathbb{E}[\vert U_n\vert] \ .
\end{equation}
According to Lemma~\ref{L:mini} we have
$(1-O(\frac{1}{n}))\,\mathbb{E}[\vert U_n\vert]<
\vert U_n\vert$ a.s.~and therefore
\begin{equation}\label{E:jk1}
(1-(1+o(1)+O(\frac{1}{n}))\,
\pi(\chi_n))\,\vert \Gamma_n\vert \le \vert U_n\vert
\qquad \text{\rm a.s.}
\end{equation}
Eq.~(\ref{E:jk0}) and eq.~(\ref{E:jk1}) imply
\begin{equation}
(1-o(1))\, \pi(\chi_n) \vert \Gamma_n\vert \le \vert \Gamma_{n,k}\vert
\le (1+o(1))\, \pi(\chi_n) \vert \Gamma_n\vert\qquad \text{\rm a.s.,}
\end{equation}
whence the lemma.
\end{proof}

Finally we show that $\Gamma_{n,k}$ is a.s.~$2$-dense 
in $Q_2^n$ with the exception of $2^n\, e^{-\tilde{\Delta}\, n^\delta}$ 
vertices. Accordingly $\Gamma_{n,k}$ is uniformly distributed in $\Gamma_n$.
The lemma will allow us to establish via Lemma~\ref{L:split} the existence of 
many vertex disjoint short paths between certain splits of the 
$\Gamma_{n,k}$-vertices.  
\begin{lemma}\label{L:dense-0}
Let $k\in\mathbb{N}$ and $\lambda_n=\frac{1+\chi_n}{n}$ and
$\epsilon\ge \chi_n\ge  n^{-\frac{1}{3}+\delta}$. Then we have
\begin{eqnarray}\label{E:there-1}
\exists\,\Delta>0;\, \forall \, v\in\mathbb{F}_2^n,\quad 
\mathbb{P}\left(\vert\text{\sf S}(v,2)\cap\Gamma_{n,k}\vert   <  
\frac{1}{2} \left(\frac{k}{2(k+1)}\right)^2 \, n^\delta\right)  
& \le &  e^{-\Delta\, n^\delta}
\end{eqnarray}
Let $D_\delta=\{v\mid \vert \text{\sf S}(v,2)\cap \Gamma_{n,k}\vert <
\frac{1}{2}\left(\frac{k}{2(k+1)}\right)^2n^\delta\}$,
then 
\begin{equation}\label{E:D-0} 
\vert D_\delta\vert  \le  2^n\, e^{-\tilde{\Delta}\, n^\delta}   
\quad \text{\it a.s., where}\quad \Delta>\tilde{\Delta}>0\ .
\end{equation}
\end{lemma}
\begin{proof}
To prove the lemma we use the last (eq.~(\ref{E:not-seq}))
$\iota_n=\lfloor \frac{k}{2(k+1)}u_nn\rfloor$ elements 
$e_{1}^{(k+1)},\dots, e_{\iota_n}^{(k+1)}$. We consider 
for arbitrary $v\in Q_2^n$
\begin{equation}
\text{\sf S}^{(k+1)}(v,2)=
\{ v + e_i^{(k+1)}+  e_j^{(k+1)}\mid 1\le i<j \le \iota_n, \,\} \ .
\end{equation}
Clearly, $\vert \text{\sf S}^{(k+1)}(v,2)\vert =\binom{\iota_n}{2}$ 
holds. By construction the $\Gamma_n$-(sc) of size 
$\ge c_k \, (u_n n)\varphi_n^k$ of Lemma~\ref{L:poly} are vertex 
disjoint for any two vertices in 
$\text{\sf S}^{(k+1)}(v,2)\cap \Gamma_n$ and each $\Gamma_n$-vertex 
belongs to $\Gamma_{n,k}$ with probability $\ge \pi_k(\chi_n)$. 
Let $Z$ be the r.v.~counting the number of vertices in 
$\text{\sf S}^{(k+1)}(v,2)\cap
\Gamma_{n,k}$. Then we have $\mathbb{E}[Z]\sim 
\left(\frac{k}{2(k+1)}\right)^2 \,\frac{u_n^2}{2}\, n\, \pi(\chi_n) 
$.
Eq.~(\ref{E:there-1}) follows now from eq.~(\ref{E:cher}), 
$u_n^2n\chi_n\ge n^\delta$ and
$
\mathbb{P}(\vert \text{\sf S}(v,2)\cap \Gamma_{n,k}\vert<\eta ) \, \le \,  
\mathbb{P}(\vert \text{\sf S}^{(k+1)}(v,2)\cap \Gamma_{n,k}\vert<
\eta) 
$.
Let now $D_\delta=\{v\mid \vert \text{\sf S}(v,2)\cap \Gamma_{n,k}\vert <
\frac{1}{2}\left(\frac{k}{2(k+1)}\right)^2n^\delta\}$.
By linearity of expectation 
$\mathbb{E}(\vert D_\delta\vert)\le 2^ne^{-\Delta\,n^\delta}$ holds 
and using Markov's inequality, $\mathbb{P}(X>\,t\mathbb{E}(X))\le 1/t$ 
for $t>0$, we derive that $\vert D_\delta\vert \le 2^ne^{-\tilde{\Delta} 
n^\delta}$ a.s.~for any $0<\tilde{\Delta}<\Delta$. 
\end{proof}


\section{Vertex boundary and paths}

The following Proposition is due to \cite{Babai:85} used for Sidon sets 
in groups in the context of Cayley graphs. In the following $G$ denotes
a finite group and $M$ a finite set acted upon by $G$.
\begin{proposition}\label{P:sidon}
Suppose $G$ act transitively on $M$ and let $A\subset M$, then we have
\begin{equation}\label{E:cool1}
\frac{1}{\vert G\vert}\sum_{g\in G}\vert A\cap gA\vert =
\vert A\vert^2/\vert M\vert \ .
\end{equation}
\end{proposition}
\begin{proof}
We prove eq.~(\ref{E:cool1}) by induction on $\vert A\vert$. 
For $A=\{x\}$ we derive $\frac{1}{\vert G\vert}\sum_{gx=x}1 = 
\vert G_x\vert/\vert G\vert$, since
$\vert M\vert=\vert G\vert/\vert G_x\vert$. We next prove the induction 
step. We write $A=A_0\cup \{x\}$ and compute
\begin{eqnarray*}
\frac{1}{\vert G\vert} \sum_g \vert A\cap gA\vert & = & 
\frac{1}{\vert G\vert} \sum_g (\vert A_0\cap gA_0\vert+
\vert \{gx\}\cap A_0\vert + \vert\{x\}\cap gA_0\vert +
\vert \{gx\}\cap\{x\} \vert \\
& = & \frac{1}{\vert G\vert}(
\vert A_0\vert^2 \vert G_x\vert + 
2 \vert A_0\vert \vert G_x\vert +\vert G_x\vert) \\
& = & \frac{1}{\vert G\vert}((\vert A_0\vert+1)^2 \vert G_x\vert) =
\frac{\vert A\vert^2}{\vert M\vert} \ .
\end{eqnarray*}
\end{proof}

Aldous \cite{Aldous:87,Babai:91b} observed how to use 
Proposition~\ref{P:sidon} for deriving a very general lower bound
for vertex boundaries in Cayley graphs:
\begin{lemma}\label{L:bound}
Suppose $G$ acts transitively on $M$ and let $A\subset M$, and let 
$S$ be a generating set of the Cayley graph $\text{\sf Cay}(G,S)$
where $\vert S\vert=n$. 
Then we have
\begin{equation}
\exists \, s\in S;\quad \vert sA\setminus A\vert \ge 
\frac{1}{n}
\vert A\vert (1-\frac{\vert A\vert}{\vert M\vert}) \ .
\end{equation}
\end{lemma}
\begin{proof}
We compute
\begin{equation}
\vert A\vert  =  
\frac{1}{\vert G\vert}\sum_g(\vert gA\setminus A\vert +\vert A\cap gA\vert) 
 =  \frac{1}{\vert G\vert}\sum_g\vert gA\setminus A\vert +\vert 
A\vert \frac{\vert A\vert}{\vert M\vert} 
\end{equation} 
and hence 
$
\vert A\vert (1-\frac{\vert A\vert}{\vert M\vert})=
\frac{1}{\vert G\vert}\sum_g\vert gA\setminus A\vert  
$.
From this we can immediately conclude
$$
\exists\, g\in G;\quad \vert gA\setminus A\vert\ge 
\vert A\vert (1-\frac{\vert A\vert}{\vert M\vert}) \ .
$$
Let $g=\prod_{j=1}^k s_j$. Since each element of $gA\setminus A$ is
contained in at least one set $s_jA\setminus A$ we obtain
$$
\vert gA\setminus A\vert \le \sum_{j=1}^k \vert s_jA\setminus A\vert \ .
$$
Hence there exists some $1\le j\le k$ such that
$\vert s_jA\setminus A\vert\ge 
\frac{1}{k}\vert gA\setminus A\vert$ and the lemma
follows.
\end{proof}

The next lemma proves the existence of many vertex disjoint paths connecting
the boundaries of certain splits of $\Gamma_{n,k}$-vertices. The lemma is
related to a result in \cite{Spencer:03} but much stronger since 
the actual length of these paths is $\le 3$. The shortness of these paths
results from the $2$-density of $\Gamma_{n,k}$ (Lemma~\ref{L:dense-0}) and
is a consequence of our particular construction of small subcomponents in
Lemma~\ref{L:poly}.
\begin{lemma}\label{L:split}
Suppose $\lambda_n=\frac{1+\chi_n}{n}$ where $\chi_n=\epsilon$ or 
$o(1)=\chi_n\ge n^{-\frac{1}{3}+\delta}$
Let $(A,B)$ be a split of the $\Gamma_{n,k}$-vertex set with the properties
\begin{equation}\label{E:prop}
\exists\,0< \sigma_0\le \sigma_1<1;\quad 
\frac{1}{n^2}\, 2^n\le \vert A\vert = \sigma_0 \vert \Gamma_{n,k}\vert  \quad 
\text{\rm and}\quad  
\frac{1}{n^2}\, 2^n\le \vert B \vert = \sigma_1 \vert \Gamma_{n,k}\vert \ .
\end{equation} 
Then there exists some $t>0$ such that a.s.~$\text{\sf d}(A)$ is 
connected to $\text{\sf d}(B)$ in $Q_2^n$ via at least 
\begin{equation}
\frac{t}{n^4}  \, 2^n/\binom{n}{7}
\end{equation}
vertex disjoint (independent) paths of length $\le 3$.
\end{lemma}
\begin{proof}
We consider $\text{\sf B}(A,2)$ and distinguish the cases
\begin{equation}\label{E:distin}
\vert \text{\sf B}(A,2)\vert\le \frac{2}{3} \, 2^n \quad \text{\rm and}
\quad
\vert\text{\sf B}(A,2)\vert>\frac{2}{3}\, 2^n \ .
\end{equation}
Suppose first $\vert \text{\sf B}(A,2)\vert \le \frac{2}{3} \, 2^n$ holds.
According to Lemma~\ref{L:bound} and eq.~(\ref{E:prop}) we have
\begin{equation}
\exists \,d_1>0;\quad \vert {\sf d}({\sf B}(A,2))\vert 
\ge \frac{d_1}{n^3} \, 2^n  
\end{equation}
and Lemma~\ref{L:dense-0} guarantees that a.s.~all except of at most 
$2^n\, e^{-\tilde{\Delta} n^\delta}$ $Q_2^n$-vertices are within distance 
$2$ to 
some $\Gamma_{n,k}$-vertex. Hence there exist at least $\frac{d}{n^3} 
\, 2^n$ vertices of $\text{\sf d}(\text{\sf B}(A,2))$ (which by definition 
are not contained in $\text{\sf B}(A,2)$) contained in $\text{\sf B}(B,2)$ 
i.e.
\begin{equation}
\vert \text{\sf d}\text{\sf B}(A,2)\cap \text{\sf B}(B,2)\vert\ge 
\frac{d}{n^3} \, 2^n
\quad \text{\rm a.s.}
\end{equation}
For each $\beta_2 \in{\sf d}(\text{\sf B}(A,2))\cap \text{\sf B}(B,2)$ 
there exists a path $(\alpha_1,\alpha_2,\beta_2)$, starting in 
$\text{\sf d}(A)$ with terminus $\beta_2$. In view of 
$\text{\sf B}(B,2)={\sf d}(\text{\sf B}(B,1))\dot\cup 
\text{\sf B}(B,1)$ we distinguish the following cases
\begin{equation}
\vert {\sf d}(\text{\sf B}(A,2))\cap {\sf d}(\text{\sf B}(B,1))
\vert\ge \frac{1}{n^3} d_{2,1} \, 2^n \quad 
\text{\rm and}\quad
\vert {\sf d}(\text{\sf B}(A,2))\cap \text{\sf B}(B,1)
\vert\ge \frac{1}{n^3} \, d_{2,2} \, 2^n \ .
\end{equation}
Suppose we have $\vert {\sf d}(\text{\sf B}(A,2))\cap 
{\sf d}(\text{\sf B}(B,1))\vert\ge \frac{1}{n^3} d_{2,1} \, 2^n $.
For each $\beta_2\in {\sf d}(\text{\sf B}(B,1))$ we select some element 
$\beta_1(\beta_2) \in \text{\sf d}(B)$ and set $B^*\subset \text{\sf d}(B)$ 
to be the set of these endpoints. Clearly at most $n$ elements in 
$\text{\sf B}(B,2)$ can produce the same endpoint, whence
$$
\vert B^*\vert\ge \frac{1}{n^4} d_{2,1} \, 2^n \  .
$$
Let $B_1\subset B^*$ be maximal subject to the condition that for any 
pair of $B_1$-vertices $(\beta_1,\beta_1')$ we have $d(\beta_1,\beta_1')> 
6$. Then we have $\vert B_1\vert \ge \vert B^*\vert/ \binom{n}{7}$ since 
$\vert \text{\sf B}(v,7)\vert=\binom{n}{7}$. 
Any two of the paths from $\text{\sf d}(A)$ to $B_1\subset \text{\sf d}(B)$ 
are of the form $(\alpha_1,\alpha_2,\beta_2,\beta_1)$ and vertex disjoint 
since each of them is contained in $\text{\sf B}(\beta_1,3)$. Therefore 
there are a.s.~at least
\begin{equation}
\frac{1}{n^4} d_{2,1} \, 2^n /\binom{n}{7}
\end{equation}
vertex disjoint paths connecting $\text{\sf d}(A)$ and $\text{\sf d}(B)$. 
Suppose next 
$\vert {\sf d}(\text{\sf B}(A,2))\cap \text{\sf B}(B,1)
\vert\ge \frac{1}{n^3} \, d_{2,2} \, 2^n$. We conclude in complete 
analogy that there exist a.s.~at least 
\begin{equation}
\frac{1}{n^3} d_{2,2} \, 2^n /\binom{n}{5}
\end{equation}
vertex disjoint paths of the form $(\alpha_1,\alpha_2,\beta_2)$ connecting 
$\text{\sf d}(A)$ and $\text{\sf d}(B)$. 
It remains to consider the case $\vert \text{\sf B}(A,2)\vert > 
\frac{2}{3}2^n$. By construction both $A$ and $B$ satisfy 
eq.~(\ref{E:prop}), respectively, whence it suffices to assume that also 
$\vert \text{\sf B}(B,2)\vert > \frac{2}{3}2^n$ holds. 
In this case we have
$$ 
\vert \text{\sf B}(A,2)\cap \text{\sf B}(B,2)\vert > \frac{1}{3}\,2^n
$$
and to each $\alpha_2\in\text{\sf B}(A,2)\cap \text{\sf B}(B,2)$ we select
$\alpha_1\in \text{\sf d}(A)$ and $\beta_1\in \text{\sf d}(B)$. We derive
in analogy to the previous arguments that there exist a.s.~at least
\begin{equation}
\frac{1}{n^2}d_2 \, 2^n /\binom{n}{5}
\end{equation}
pairwise vertex disjoint paths of the form $(\alpha_1,\alpha_2,\beta_1)$
and the proof of the lemma is complete.
\end{proof}


\section{The largest component}

\begin{theorem}\label{T:giant}
Let $Q_{2,\lambda_n}^n$ be the random graph consisting of $Q_2^n$-subgraphs,
$\Gamma_n$, induced by selecting each $Q_2^n$-vertex with independent 
probability $\lambda_n$. Suppose $\lambda_n=\frac{1+\chi_n}{n}$, where 
$\epsilon\ge \chi_n\ge n^{-\frac{1}{3}+\delta}$, $\delta>0$. Then we have
\begin{equation}\label{E:large} 
\lim_{n\to\infty}\mathbb{P}\left(\,  
\vert C_n^{(1)}\vert \sim \pi(\chi_n)\,\frac{1+\chi_n}{n}\,2^n
\ \text{\it and $C_n^{(1)}$ is unique}\, \right)= 1 \ .
\end{equation}
\end{theorem}
\begin{proof}
{\it Claim.} We have $\vert C_n^{(1)}\vert \sim \vert \Gamma_{n,k}\vert$ a.s.\\
To prove the Claim we use an idea introduced by Ajtai {\it et.al.} 
\cite{Ajtai:82} and select $Q_2^n$-vertices in two rounds. 
First we select $Q_2^n$-vertices with independent probability 
$\frac{1+\chi_n/2}{n}$ and subsequently with $\frac{\chi_n}{2n}$. 
The probability for some vertex not to be chosen
in both randomizations is $(1-\frac{1+\chi_n/2}{n})
(1-\frac{\chi_n/2}{n})=1-\frac{1+\chi_n}{n}+\frac{(1+\chi_n/2)
\chi_n/2}{n^2}\ge 1-\frac{1+\chi_n}{n}$. Hence selecting first 
with probability $\frac{1+\chi_n/2}{n}$ (first round) and then with 
$\frac{\chi_n/2}{n}$ (second round) a vertex is selected with probability 
less than $\frac{1+\chi_n}{n}$ (all preceding lemmas hold for the first
randomization $\frac{1+\chi_n/2}{n}$).
We now select in our first round each $Q_2^n$-vertex with probability 
$\frac{1+\chi_n/2}{n}$. According to Lemma~\ref{L:size}
\begin{equation}
\vert \Gamma_{n,k}\vert\sim \pi(\chi_n)\, \vert\Gamma_n\vert 
\quad \text{\rm a.s.}
\end{equation}
Suppose $\Gamma_{n,k}$ contains a component, $A$, such that
$$
\frac{1}{n^2}\, 2^n \le \vert A\vert \le (1-b)\,
\vert \Gamma_{n,k}\vert, \quad b>0
$$
then there exists a split of $\Gamma_{n,k}$, $(A,B)$ satisfying the assumptions
of Lemma~\ref{L:split} 
(and $\text{\sf d}(A)\cap \text{\sf d}(B)=\varnothing$). We now observe that
Lemma~\ref{L:poly} limits the number of ways these splits can be 
constructed. In view of
\begin{equation}
\lfloor \frac{1}{4}u_nn\rfloor\varphi_n^k 
   \ge c_k \, n^{\frac{2}{3}} n^{k\delta},\quad c_k>0
\end{equation}
each $A$-vertex is contained in a component of size 
at least $c_k \, n^{\frac{2}{3}} n^{k\delta}$.
Therefore there are at most 
\begin{equation}
2^{\left(2^n/(c_k \, n^{\frac{2}{3}} n^{k\delta})\right)}
\end{equation}
ways to choose $A$ in such a split. According to Lemma~\ref{L:split}
there exists $t>0$ such that a.s.~$\text{\sf d}(A)$ is 
connected to $\text{\sf d}(B)$ in $Q_2^n$ via at least 
$\frac{t}{n^4}\, 2^n/\binom{n}{7}$ vertex disjoint paths of 
length $\le 3$.
We now select $Q_2^n$-vertices with probability $\frac{\chi_n/2}{n}$. 
None of the above $\ge \frac{t}{n^4}\, 2^n/\binom{n}{7}$ paths can be selected 
during this process. Since any two paths are vertex disjoint the expected 
number of such splits is less than
\begin{equation}
2^{\left(2^n/(c_k \, n^{\frac{2}{3}} n^{k\delta})\right)}\, 
\left(1-\left(\frac{\chi_n/2}{n}\right)^4\right)^{ \frac{t}{n^4}\, 
2^n/\binom{n}{7}} \sim 
2^{\left(2^n/(c_k \, n^{\frac{2}{3}} n^{k\delta})\right)}\,
e^{-\frac{t\chi_n^4}{2^4n^8}\, 2^n/\binom{n}{7}} \ .
\end{equation}
Hence choosing $k$ sufficiently large, we can conclude that
a.s.~there cannot exist such a split. 
Therefore $\vert C_n^{(1)} \vert \sim \vert \Gamma_{n,k}\vert$, 
a.s.~and the Claim is proved.
According to Lemma~\ref{L:size} we therefore have $\vert C_n^{(1)}\vert 
\sim \pi(\chi_n)\,\vert \Gamma_n\vert$. In particular, for $\chi_n=\epsilon$, 
Theorem~\ref{T:galton} ($0<\alpha(\epsilon)<1$) implies that there 
exists a giant component.
It remains to prove that $C_n^{(1)}$ is unique. 
By construction any large component, $C_n'$, is necessarily contained in 
$\Gamma_{n,k}$. In the proof of the Claim we have shown that a.s.~there 
cannot exist a component $C_n'$ in $\Gamma_n$ with the property 
$\vert C_n'\vert \ge \frac{1}{n^2}\,\vert \Gamma_n\vert$. 
Therefore $C_n^{(1)}$ is unique and the proof of the theorem is complete.
\end{proof}

Theorem~\ref{T:giant-2} below is the analogue of Ajtai {\it et.al.}'s result 
\cite{Ajtai:82} (for random subgraphs of $n$-cubes obtained by 
selecting $Q_2^n$-edges independently).
\begin{theorem}\label{T:giant-2}
Let $Q_{2,\lambda_n}^n$ be the random graph consisting of $Q_2^n$-subgraphs,
$\Gamma_n$, induced by selecting each $Q_2^n$-vertex with independent 
probability $\lambda_n$. Then 
\begin{equation}\label{E:giant-alt}
\lim_{n\to\infty}\mathbb{P}(\Gamma_n \,\text{\it has an unique giant 
component})=
\begin{cases}
1 & \text{\rm for $\lambda_n\ge \frac{1+\epsilon}{n}$ }\\
0 & \text{\rm for $\lambda_n\le \frac{1-\epsilon}{n}$ }.
\end{cases} 
\end{equation}
\end{theorem}
\begin{proof}
We proved the first assertion in Theorem~\ref{T:giant}.
It remains to consider the case $\lambda_n=\frac{1-\epsilon}{n}$.\\
{\it Claim.} Suppose $\lambda_n=\frac{1-\epsilon}{n}$, then there
exists $\kappa'>0$ such that $\vert C_n^{(1)}\vert \le \kappa'\,n$ holds.\\
The expected number of components of size $\ell$ is less than
\begin{equation}
\frac{1}{\ell}\, 2^n\, n^{\ell-1}\,\left(\frac{1-\epsilon}{n}\right)^{\ell}=
\frac{1}{\ell\,n} \, 2^n\,(1-\epsilon)^{\ell}
\end{equation}
since there are $2^n$ ways to choose the first element and at most
$n$-vertices to choose from subsequently. This component is counted
$\ell$ times corresponding to all $\ell$ choices for the ``first'' 
vertex. Let $X_{\kappa'\, n}$ be the r.v.~counting the number
of components of size $\ge \kappa'\,n$. Choosing $\kappa'$ such 
that $(1-\epsilon)^{\kappa'}<1/4$ we obtain
\begin{equation}
\mathbb{E}(X_{\kappa'\,n})\le \sum_{\ell\ge \kappa'\, n}
\frac{1}{\ell\,n} \, 2^n\,(1-\epsilon)^{\ell}\le
\frac{1}{n^2}2^n (1-\epsilon)^{\kappa'\, n}\,
\sum_{\ell\ge 0}(1-\epsilon)^\ell<
\frac{1}{n^2}\, \left(\frac{1}{2}\right)^{n}\,
\frac{1}{1-(1-\epsilon)} \ ,
\end{equation}
whence the Claim and the proof of the theorem is complete.
\end{proof}

{\bf Acknowledgments.}
We thank E.Y.~Jin, J.~Qin and L.C.~Zuo for helpful suggestions.
Special thanks to the referees for helpful comments.
This work was supported by the 973 Project, the PCSIRT Project of the
Ministry of Education, the Ministry of Science and Technology, and
the National Science Foundation of China.

\bibliographystyle{plain}


\end{document}